\documentclass[12pt]{amsart}
\usepackage[T1]{fontenc}
\usepackage[english]{babel}
\usepackage{amsmath, amssymb, amsthm, amscd,color,comment}
\usepackage{cancel}
\usepackage{cite}
\usepackage{alltt}
\usepackage[dvipsnames]{xcolor}
\usepackage{array}
\usepackage[small,bf,labelsep=period]{caption}
\usepackage{mathtools}
\usepackage{hyperref}

\usepackage{pgfplots}
\pgfplotsset{compat=1.15}
\usepackage{mathrsfs}
\usetikzlibrary{arrows}

\usepackage{enumerate}
\newtheorem{theorem}{Theorem}[section]

\newtheorem{example}[theorem]{Example}

\newtheorem{remark}[theorem]{Remark}
\newtheorem{lemma}[theorem]{Lemma}
\newtheorem{corollary}[theorem]{Corollary}
\newtheorem{proposition}[theorem]{Proposition}

\DeclarePairedDelimiter\ceil{\lceil}{\rceil}
\DeclarePairedDelimiter\floor{\lfloor}{\rfloor}

\def\la{\lambda}
\def\a{\alpha}

\def\N{\mathbb N}
\def\R{\mathbb R}

\def\cB{\mathcal B}
\def\cC{\mathcal C}

\def\cH{\mathcal H}

\def\cL{\mathcal L}

\def\cP{\mathcal P}

\def\cX{\mathcal X}
\def\cY{\mathcal Y}

\def\a{\alpha}

\def\fqss{\mathbb F_{q^6}}
\def\fqs{\mathbb F_{q^2}}

\def\fqsn{{\mathbb F}_{q^{2n}}}

\def\Ap{{\rm Ap}}

\def\Div{{\rm Div}}
\def\O{{\rm O}}

\def\deg{{\rm deg}}

\def\BW{{\rm BW}}
\def\wt{{\rm wt}}

\def\a{\alpha}
\def\char{\mbox{\rm Char}}

\newcommand{\al}{\alpha}
\newcommand{\be}{\beta}

\begin{document}

\title[On gap sets in arbitrary Kummer extensions of $K(x)$]{On gap sets in arbitrary Kummer extensions of $K(x)$}



%\thanks{The second author was partially supported by FAPESP Grant 2022/16369-2.}

\author{Ethan Cotterill, Erik A. R. Mendoza, and Pietro Speziali}

\address{Universidade Estadual de Campinas, Instituto de Matemática, Estatística e Computação Científica, CEP 13083-859, Campinas, Brazil}
\email{ethan@unicamp.br}
\address{Universidade Federal de Viçosa, Departamento de Matemática, CEP 36570-900, Viçosa, Brazil}
\email{erik.mendoza@ufv.br}
\address{Universidade Estadual de Campinas, Instituto de Matemática, Estatística e Computação Científica, CEP 13083-859, Campinas, Brazil}
\email{speziali@unicamp.br}

\begin{abstract} 
Let $K$ be an algebraically closed field, and let $F/K(x)$ be a Kummer extension of function fields of genus $g$. We provide a compact and explicit description of the gap set $G(Q)$ at any totally ramified place $Q$ of the extension $F/K(x)$. As a consequence, we deduce structural properties of the Weierstrass semigroup $H(Q)$; in particular, we determine a generating set for $H(Q)$, and we characterize its symmetry in certain cases. We also generalize a formula due to Towse \cite{T1996} that describes the asymptotic behavior of the sum of the Weierstrass weights at all totally ramified places of the extension $F/K(x)$ relative to $g^3-g$. 
\end{abstract}

\maketitle

\section{Introduction}

Let $F$ be an algebraic function field in one variable of genus $g$ over an algebraically closed field $K$. Given a place $Q$ in $F$, an integer $n\in \N_0$ is a {\it non-gap} at $Q$ whenever there exists some $z\in F$ with pole divisor $(z)_\infty = nQ$. The set of non-gaps at $Q$ constitutes the {\it Weierstrass semigroup} at $Q$ and is denoted hereafter by $H(Q)$. The elements of the complement $G(Q)=\N_0\setminus H(Q)$ are {\it gaps} at $Q$. %According to the Weierstrass Gap Theorem \cite[Theorem 1.6.8]{S2009}, 
Riemann--Roch implies that every place in $F$ has precisely $g$ gaps. %Additionally, it is known that 
Moreover, the gap set is the same for all but finitely many places of $F$; this common set as the {\it gap sequence of $F$}. Whenever the gap sequence of $F$ is $\{1, 2, \dots, g\}$, the function field $F$ is {\it classical}. Places where the gap set differs from the gap sequence of $F$ are {\it Weierstrass places}.

Weierstrass semigroups of function fields %and their associated gap sets 
are classical objects of study in algebraic geometry, in part because of their diverse applications. %Among the most significant of these are 
They play a decisive role in determining of the automorphism group of an algebraic curve \cite{MXY2016, MTZ2024}; in establishing upper bounds for the number of rational points on an algebraic curve over a finite field \cite{L1990, GM2009, BR2013}; in constructing Goppa codes with good parameters \cite{G1977, GKL1993, FR1993, MST2009}; and, more recently, in constructing families of non-isomorphic maximal curves with equal genus and automorphism groups \cite{BMNQ2025}.

Determining the Weierstrass semigroup or its associated gap set at a place is a challenging problem in general. Recently, several authors have determined Weierstrass semigroups for specific curves; see, e.g., \cite{BM2018-II, MTZ2018, BMZ2021, MP2020, BLM2021, BMV2023}. On the other hand, in \cite{ABQ2019} the authors provided an arithmetic criterion to determine whether a given positive integer %is an element of the gap set 
belongs to $G(Q)$, whenever $Q$ is a totally ramified place of %an arbitrary Kummer extension 
a cyclic cover of the projective line %defined by the equation 
$y^m = f(x)$, with $f(x) \in K[x]$ a separable polynomial. As a consequence, they explicitly described the semigroup $H(Q)$. %when $f(x)$ is a separable polynomial. 
This description was subsequently generalized in \cite{CMQ2016}, where the authors study the Kummer extension of $K(x)$ defined by $y^m = f(x)^{\lambda}$, where $\lambda\in \N$ and $f(x)\in K[x]$ is a separable polynomial with $\gcd(m,\lambda \cdot \deg f) = 1$.

In \cite{M2023}, the second author studied arbitrary Kummer extensions with one place at infinity $Q_0$, explicitly determined the semigroup $H(Q_0)$ and gap set $G(Q_0)$, and proved several results about the structure of $H(Q_0)$. In collaboration with Castellanos and Quoos \cite{CMQ2024}, using a different approach, he subsequently determined the gap set at any totally ramified place of an arbitrary Kummer extension of $K(x)$ with one place at infinity. %was explicitly determined. 
In the current work, we generalize these results by providing a compact and explicit description of the gap set $G(Q)$ at any totally ramified place $Q$ of an arbitrary Kummer extension of $K(x)$.
%the projective line. 
We also construct bases of holomorphic differentials that allow us to derive these gap sets and, furthermore, determine gaps at unramified and partially ramified places of an arbitrary Kummer extension of $K(x)$. As a consequence of our description of $G(Q)$, we deduce several extensions of results about $H(Q)$ obtained in \cite{M2023}.

For a classical function field $F$ with genus $g$ defined over an algebraically closed field $K$ such that $\char(K)=0$ or $\char(K)\geq 2g-2$, the Weierstrass weight of a place $Q$ of $F$ is $\wt(Q)=\sum_{s\in G(Q)}s-g(g+1)/2$; and it is well-known that the number of Weierstrass places of $F$, counted with their weights, is $g^3-g$. In studying the Weierstrass places of a Kummer extension $K(x, y)/K(x)$ defined by the equation $y^m = f(x)$, where $f(x)\in K[x]$ is a separable polynomial of degree $r$, Towse \cite{T1996} established a lower bound for the limit
$$
\lim_{r \to \infty} \frac{\BW}{g^3-g}
$$
where $\BW$ is the sum of the Weierstrass weights of all totally ramified places of the extension $K(x, y)/K(x)$, and calculated its exact value assuming that $\gcd(m, r)=1$. Towse's result was generalized in \cite{ABQ2019}, where the exact value of the limit was determined without assuming $\gcd(m, r)=1$. Here we generalize these results by providing an exact formula for $\lim_{r \to \infty}\BW/(g^3-g)$ in Kummer extensions of $K(x)$ under certain conditions.

This paper is organized as follows. In Section 2, we review background and introduce notation that will be used throughout the work. In Section 3, we provide a compact and explicit description of the gap set $G(Q)$ at any totally ramified place $Q$ of an arbitrary Kummer extension $F/K(x)$, see Theorem \ref{teo_gapsets}; and we determine gaps at unramified and partially ramified places of $F/K(x)$, see Remark \ref{Remark_basis} and Proposition \ref{prop_gaps_2}, respectively. As consequences, we derive a description of the multiplicity of the semigroup $H(Q)$, see Corollaries \ref{coro3_multiplicidade} and \ref{coro_multiplicity2} and Proposition \ref{prop_mult-frob}; provide sufficient conditions for $H(Q)$ to be symmetric, see Corollary  \ref{coro_symmetric}; compute the Apéry set for certain elements of $H(Q)$ and determine a set of generators for the semigroup $H(Q)$, see Corollary \ref{coro2_generators}. We also completely characterize the symmetry of the semigroup $H(Q)$ under some conditions; see Theorems \ref{teo_sym1} and \ref{teo_sym2}. In Section 4, we generalize the formula given by Towse \cite{T1996} for the asymptotic behavior of the sum of the Weierstrass weights at all totally ramified places of the extension $F/K(x)$ relative to $g^3-g$; see Theorem \ref{teo_limit}.

\section{Background and notation}
Throughout this article, %for $a$ and $b$ integers, we denote by 
$(a, b)$ denotes the greatest common divisor of integers $a$ and $b$, while $b \bmod{a}$ denotes the smallest non-negative integer congruent with $b$ modulo $a$. Given $c\in \R$, we let $\floor*{c}$, $\ceil*{c}$ and $\{c\}$ denote the floor, ceiling and fractional part functions of $c$, respectively. We also let $\N_0 = \N \cup \{0\}$, where $\N$ is the set of positive integers. 

\subsection{Numerical semigroups}
A {\it numerical semigroup} is a subset $H$ of $\N_0$ that is closed under addition, contains $0$, and whose complement $\N_0\setminus H$ is finite. The elements of $G:=\N_0\setminus H$ are the {\it gaps} of the numerical
semigroup $H$ and $g_H:=\#G$ is its {\it genus}. The largest gap is the {\it Frobenius number} of $H$ and is denoted by $F_{H}$. The smallest nonzero element of $H$ is the {\it multiplicity} of the semigroup and is denoted by $m_H$. The numerical semigroup $H$ is {\it symmetric} whenever $F_H=2g_H-1$. Moreover, a subset $\{a_1, \dots, a_n\}\subset H$ is a {\it system of generators} for $H$ whenever
$$
H=\langle a_1, \dots, a_n\rangle:=\{t_1a_1+\cdots + t_na_n: t_1, \dots, t_n\in \N_0\}.
$$ 

\noindent Given a nonzero element $m$ of the numerical semigroup $H$. The
{\it $m$-Apéry set} of $H$ is 
$$
\Ap(H,m):=\{0=w(0),w(1),\dots,w(m-1)\}
$$
where $w(i)$ is the least element of
$H$ congruent to $i$ modulo $m$. Many properties of numerical semigroups may be framed in terms of Apéry sets. These include the following:
\begin{proposition}\cite[Lem. 2.6]{RG2009}\label{prop_Apery-generators}
Let $m$ be a nonzero element of a numerical semigroup $H$. The set $\{m\}\cup (\Ap(H, m)\setminus \{0\})$ is a system of generators for $H$.
\end{proposition}
\begin{proposition}\cite[Prop. 4.10]{RG2009}\label{prop_Apery-symmetric}
Let $m$ be a nonzero element of a numerical semigroup $H$; and let $\Ap (H,m) = \{0=a_0 < a_1 <\dots< a_{m-1}\}$ be the $m$-Apéry set of $H$. The numerical semigroup $H$ is symmetric if and only if 
$$
a_i+a_{m-1-i}= a_{m-1}\quad \text{ for }\quad i=1, \dots, m-1.
$$
\end{proposition}
\subsection{Function fields and Weierstrass semigroups}

Let $K$ be an algebraically closed field, and let $F/K$ be a function field of one variable of genus $g(F)$. We let $\mathcal P_{F}$ denote the set of places in $F$, $\Omega_{F}$ the space of differentials forms on $F$, $\nu_{P}$ the discrete valuation of $F/K$ determined by the place $P\in \cP_{F}$, and  $\Div (F)$ the group of divisors on (the smooth algebraic curve over $K$ determined by) $F$. Given a function $z \in F$, we let $(z)_{F}, (z)_\infty,$ and $(z)_0$ denote the principal, pole, and zero divisors of the function $z$ in $F$, respectively. 

Given a divisor $G\in \Div(F)$, the {\it Riemann-Roch space} associated to the divisor $G$ is
$$
\cL(G)=\{z\in F: (z)_{F}+G\geq 0\}\cup \{0\}
$$ 
and the space of {\it holomorphic differentials on $F$} is
$$
\Omega_F(0)=\{\omega\in \Omega_F : (\omega)_F\geq 0\}\cup \{0\}.
$$ 
It is well-known that $\Omega_F(0)$ is a vector space over $K$ with dimension $g(F)$. On the other hand, given a place $Q\in \cP_{F}$, the Weierstrass semigroup at $Q$  is defined by
$$
H(Q)=\{s\in \N_0 : (z)_{\infty}=sQ\text{ for some }z\in F\}.
$$
A non-negative integer $s$ is a {\it non-gap} at $Q$ if $s\in H(Q)$. An element in the complement $G(Q):= \N \setminus H(P)$ is a {\it gap} at $Q$. For a function field $F/K$ of genus $g(F)>0$, the number of gaps is always finite; in fact, $\#G(Q)=g(F)$. In particular, $H(Q)$ is a numerical semigroup. Moreover, for all but finitely many places on $F$, the gap set is always the same. This set is {the {\it gap sequence of $F$}, and $F$ is {\it classical} whenever its gap sequence is minimal possible, i.e., $\{1, 2, \dots, g(F)\}$. Those places at which the gap set is not equal to the gap sequence of $F$ are {\it Weierstrass places}. 

In general, determining a system of generators  or gap set for a Weierstrass semigroups a complex task. Among the key ingredients for determining the gaps of a Weierstrass semigroup is the following relating gaps and vanishing orders of holomorphic differentials.

 \begin{proposition}\cite[Cor. 14.2.5]{VS2006}\label{prop_gap_criterion}
 Let $F/K$ be a function field. Let $Q$ be a place in $F$, and let $\omega$ be a holomorphic differential on $F$. Then $\nu_Q(\omega) +1$ is a gap at $Q$.
 \end{proposition}

Now, let $\cX$ be the algebraic curve defined over $K$ and given by the affine equation 
\begin{equation}\label{curveX}
\cX: \quad y^{m}:=f(x)=\prod_{k=1}^{r} (x-\a_k)^{\lambda_k}, \quad \la_k\in \N, \quad  1\leq \la_k< m %\quad \text{and} \quad \char(K) \nmid m,
\end{equation}
where $\char(K) \nmid m$, $f(x)\in K[x]$ is not a $d$-th power %(where $d$ divides $m$) 
of an element in $K(x)$, and $\al_1, \dots, \al_r\in K$ are pairwise distinct elements. Let $K(\cX)$ be the function field of $\cX$. Then $K(\cX)/K(x)$ is a Kummer extension %of the projective line 
and, by \cite[Proposition 3.7.3]{S2009}, $\cX$ is of genus
\begin{equation}\label{eq_genus}
g(\cX)=\frac{m(r-1)+2-\sum_{k=0}^{r}(m, \la_k)}{2},\quad \text{where}\quad
\la_0=\textstyle\sum_{k=1}^{r}\la_k.
\end{equation}

For every $k=1, \dots, r$, let $P_{k}$ and $P_{0}$ be the places in $K(x)$ corresponding to the zero of $x-\al_k$ and the pole of $x$, respectively. If $(m, \la_s)=1$ for some $0\leq s \leq r$ we denote by $Q_s$ the unique place in $K(\cX)$ lying over $P_s$.

The gap set at any totally ramified place of the Kummer extension $K(\cX)/K(x)$ was calculated in \cite{CMQ2024} assuming that $(m, \la_0)=1$.

\begin{proposition}\cite[Prop. 4.3]{CMQ2024}\label{prop1}
Suppose that $(m, \la_0)=1$. The gap set at the unique place at infinity $Q_{0}$ of $K(\cX)$ is given by 
$$
G(Q_{0})=\left\{mj-i\la_0 : 1\leq i\leq m-1, \, \floor*{\frac{i\la_0}{m}}+1\leq j \leq \sum_{k=1}^{r}\ceil*{\frac{i\la_k}{m}}-1\right\}. 
$$ 
\end{proposition}

\begin{proposition}\cite[Prop. 4.4]{CMQ2024}\label{prop2}
Suppose that $(m, \la_0)=(m, \la_s)=1$ for some $1\leq s\leq r$ and let $1\leq \la^s \leq m-1$ be the inverse of $\la_{s}$ modulo $m$. The gap set at the place $Q_{s}$ of $K(\cX)$ is given by
$$
G(Q_s)=\left\{mj+i : 1\leq i\leq m-1, \, 0\leq j \leq \sum_{k=1}^{r}\ceil*{\frac{i\la^s\la_k}{m}}-\floor*{\frac{i\la^s\la_0}{m}}-2\right\}. 
$$ 
\end{proposition}

\begin{remark} Via the same argument used to prove Proposition \ref{prop2}, we can in fact prove that this result holds irrespective of whether $(m, \lambda_0) = 1$.
\end{remark}
Indeed, let
$$
G:=\left\{mj+i: 1\leq i\leq m-1, \, 0\leq j \leq \sum_{k=1}^{r}\ceil*{\frac{i\la^s\la_k}{m}}-\floor*{\frac{i\la^s\la_0}{m}}-2\right\}. 
$$ 
For $mj+i\in G$, let $t$ be the unique element in $\{0, \dots, m-1\}$ such that $mj+i + t\la_{s}\equiv 0\mod{m}$. So $-i\la^s \equiv t \mod{m}$ and we get $\{\frac{t\la_k}{m}\}=\{\frac{-i \la^s\la_k}{m}\}$ for $1\leq k \leq r$. Thus
$$
\sum_{k=1}^{r}\left\{\frac{t\la_k}{m}\right\}=\sum_{k=1}^{r}\left\{\frac{-i\la^s\la_k}{m}\right\}=\sum_{k=1}^{r}\left(-\frac{i\la^s\la_k}{m}-\floor*{\frac{-i\la^s\la_k}{m}}\right)=-\frac{i\la^s\la_0}{m}+\sum_{k=1}^{r}\ceil*{\frac{i\la^s\la_k}{m}}.
$$
Since $(m, \la^s)=1$, it follows that
$$
\sum_{k=1}^{r}\left\{\frac{t\la_k}{m}\right\}=-\frac{i\la^s\la_0}{m}+\sum_{k=1}^{r}\ceil*{\frac{i\la^s\la_k}{m}} > -1-\floor*{\frac{i\la^s\la_0}{m}}+\sum_{k=1}^{r}\ceil*{\frac{i\la^s\la_k}{m}} \geq j+1 = \ceil*{\frac{i+mj}{m}}.
$$
From \cite[Corollary 3.6]{ABQ2019}, it follows that $G\subseteq G(Q_{s})$. Using the same argument as in the proof of Proposition \ref{prop1}, it can be shown that $\# G=g(\cX)$, and therefore $G(Q_{s})=G$.

Thus, Propositions \ref{prop1} and \ref{prop2} compute the gap set at any totally ramified place of the Kummer extension $K(\cX)/K(x)$. 

\section{On the gap set $G(Q_s)$}

In this section, we provide a new compact and explicit description of the gap set $G(Q)$ at any totally ramified place $Q$ of the Kummer extension $K(\mathcal{X})/K(x)$. As a consequence of this new description, we obtain several %results that describe the internal 
structural results for the semigroup $H(Q)$. We begin by introducing some notation and reviewing a previous result. 

For every $1\leq i\leq m-1$ and $0\leq s\leq r$, set
$$
t_s(i):=\left\{\begin{array}{ll}
(i\la_s)\bmod m, & \text{if }1\leq s \leq r, \\
m-(i\la_0)\bmod m, & \text{if }s=0.
\end{array}\right.
$$
Also set %and the function 
$$
\beta_0(i):=\sum_{k=1}^{r}\ceil*{\frac{i\la_k}{m}}-\floor*{\frac{i\la_0}{m}}-1.
$$
It is clear that if $(m, \la_s)=1$ for some $0\leq s \leq r$ then $0\leq \be_0(i)\leq r-1$ for every $1\leq i \leq m-1$. In addition, for every $1\leq s\leq r$ with $(m, \la_s)=1$, define
$$
\beta_s(i):=\sum_{k=1}^{r}\ceil*{\frac{i\la^s\la_k}{m}}-\floor*{\frac{i\la^s\la_0}{m}}-1
$$
where $1\leq \la^s \leq m-1$ is the inverse of $\la_s$ modulo $m$.
\begin{lemma}\label{lema1}
Let $1\leq s \leq r$ be such that $(m, \la_s)=1$. We have
$\beta_s(t_s(i))=\beta_0(i)$ for every $1\leq i\leq m-1$.
\end{lemma}
\begin{proof}
As $\la_s\la^s\equiv 1 \bmod m$, there exists $q_s\in \N_0$ for which $\la_s\la^s=mq_s+1$. Note, moreover, that $t_s(i)=i\la_s-\floor*{i\la_s/m}m$. It follows that for every $1\leq k\leq r$, we have
\begin{align*}
\ceil*{\frac{t_s(i)\la^s\la_k}{m}}&=\ceil*{\frac{(i\la_s-\floor*{i\la_s/m}m)\la^s\la_k}{m}}\\
&=\ceil*{\frac{i\la_s\la^s\la_k}{m}}-\floor*{\frac{i\la_s}{m}}\la^s\la_k\\
&=\ceil*{\frac{i(mq_s+1)\la_k}{m}}-\floor*{\frac{i\la_s}{m}}\la^s\la_k\\
&=\ceil*{\frac{i\la_k}{m}}+\la_k\left(iq_s-\floor*{\frac{i\la_s}{m}}\la^s\right).
\end{align*}
Analogously, 
$$
\floor*{\frac{t_s(i)\la^s\la_0}{m}}=\floor*{\frac{i\la_0}{m}}+\la_0\left(iq_s-\floor*{\frac{i\la_s}{m}}\la^s\right).
$$
Therefore,
\begin{align*}
\beta_{s}(t_s(i))&=\sum_{k=1}^{r}\ceil*{\frac{t_s(i)\la^s\la_k}{m}}-\floor*{\frac{t_s(i)\la^s\la_0}{m}}-1\\
&=\sum_{k=1}^{r}\left(\ceil*{\frac{i\la_k}{m}}+\la_k\left(iq_s-\floor*{\frac{i\la_s}{m}}\la^s\right)\right)-\floor*{\frac{i\la_0}{m}}-\la_0\left(iq_s-\floor*{\frac{i\la_s}{m}}\la^s\right)-1\\
&=\sum_{k=1}^{r}\ceil*{\frac{i\la_k}{m}}-\floor*{\frac{i\la_0}{m}}-1\\
&=\beta_0(i).
\end{align*}
\end{proof}

\begin{theorem}\label{teo_gapsets}
Let $0\leq s \leq r$ be such that $(m, \la_s)=1$. The gap set at $Q_s$ is given by
$$
G(Q_s)=\{mj+t_s(i) : 1\leq i\leq m-1, \, 0\leq j \leq \beta_0(i)-1\}.
$$
\end{theorem}
\begin{proof}
We first handle the case $s=0$. %To begin, observe that, from 
Applying Proposition \ref{prop1}, we have 
%we can express the gap set $G(Q_0)$ as
\begin{align*}
G(Q_{0})&=\left\{m\left(j+\floor*{\frac{i\la_0}{m}}+1 \right)-i\la_0 : 1\leq i\leq m-1, \, 0\leq j \leq \sum_{k=1}^{r}\ceil*{\frac{i\la_k}{m}}-\floor*{\frac{i\la_0}{m}}-2\right\}\\
&= \Bigg\{mj+m-\left(i\la_0-\floor*{\frac{i\la_0}{m}}m\right) : 1\leq i\leq m-1, \, 0\leq j \leq \beta_0(i)-1\Bigg\}\\
&= \Bigg\{mj+t_0(i) : 1\leq i\leq m-1, \, 0\leq j \leq \beta_0(i)-1\Bigg\}.
\end{align*}
Similarly, when $1\leq s\leq r$,  note that $\{1, 2, \dots, m-1\}=\{t_s(1), t_s(2), \dots, t_s(m-1)\}$. It now follows from Proposition \ref{prop2} and Lemma \ref{lema1} that %express the gap set $G(Q_s)$ as
\begin{align*}
G(Q_s)&=\Bigg\{mj+i : 1\leq i\leq m-1, \, 0\leq j \leq \beta_s(i)-1\Bigg\}\\
&=\Bigg\{mj+t_s(i) : 1\leq i\leq m-1, \, 0\leq j \leq \beta_s(t_s(i))-1\Bigg\}\\
&=\Bigg\{mj+t_s(i) : 1\leq i\leq m-1, \, 0\leq j \leq \beta_0(i)-1\Bigg\}.
\end{align*}
\end{proof}

\begin{corollary}\label{coro1}
Suppose that $(m, \la_s)=1$ for some $0\leq s\leq r$; then $\sum_{i=1}^{m-1}\be_0(i)=g(\cX)$.
\end{corollary}
\begin{proof}
The assertion follows directly from Theorem \ref{teo_gapsets}, together with the fact that the cardinality of the gap set at any place of $K(\cX)$ is $g(\cX)$.
\end{proof}

\begin{remark}\label{Remark_basis}
We may explicitly construct a basis of holomorphic differentials on $\cX$ %$K(\cX)$ 
that %allow us to obtain the sets of gaps 
certifies the gap sets described in Theorem \ref{teo_gapsets}. Indeed, given $\al\in K$, let $P$ be the place in $K(x)$ corresponding to the zero of $x-\al$, and let $\la:=\nu_P(f(x))$. From \cite[Proposition 3.7.3]{S2009}, we have the following principal divisors:
\begin{equation}\label{div}
\begin{array}{l}
(x-\al)_{K(\cX)}=\dfrac{m}{(m, \la)}\displaystyle\sum_{{ Q\in \cP_{K(\cX)}, \, Q|P}}Q-\frac{m}{(m, \la_0)}\displaystyle\sum_{{ Q\in \cP_{K(\cX)}, \, Q|P_{0}}}Q,\\
(y)_{K(\cX)}=\displaystyle\sum_{k=1}^{r}\frac{\la_k}{(m, \la_k)}\displaystyle\sum_{{ Q\in \cP_{K(\cX)}, \, Q|P_{k}}}Q-\frac{\la_0}{(m, \la_0)}\displaystyle\sum_{{ Q\in \cP_{K(\cX)}, \, Q|P_{0}}}Q, \text{ and}\\
(dx)_{K(\cX)}=\displaystyle\sum_{k=1}^{r}\left(\frac{m}{(m, \la_k)}-1\right)\displaystyle\sum_{{ Q\in \cP_{K(\cX)}, \, Q|P_{k}}}Q-\left(\frac{m}{(m, \la_0)}+1\right)\displaystyle\sum_{{ Q\in \cP_{K(\cX)}, \, Q|P_{0}}}Q.
\end{array}
\end{equation}
Accordingly, for $\al\in K$, $1\leq i \leq m-1$, and $0\leq j \leq \be_0(i)-1$, set %define the differential
$$ 
\omega_{i, j}(\al):=\frac{(x-\al)^jy^idx}{\prod_{k=1}^{r}(x-\al_k)^{\ceil*{\frac{i\la_k}{m}}}}.
$$
Expanding using %Using the divisors described in 
(\ref{div}), we find that its associated divisor is
\begin{align*}
\left(\omega_{i, j}(\al)\right)_{K(\cX)}&=\sum_{k=1}^{r}\left(\frac{m\left(1+\floor*{\frac{i\la_k}{m}}-\ceil*{\frac{i\la_k}{m}}\right)+t_k(i)}{(m, \la_k)}-1\right)\sum_{{ Q\in \cP_{K(\cX)}, \, Q|P_{k}}}Q\\
&\quad +\frac{mj}{(m, \la)}\sum_{{ Q\in \cP_{K(\cX)}, \, Q|P}}Q+ \left(\frac{m(\beta_0(i)-1-j)+t_0(i)}{(m, \la_0)}-1\right)\sum_{{ Q\in \cP_{K(\cX)}, \, Q|P_{0}}}Q.
\end{align*}
As this divisor is clearly effective, it follows that for every $\al\in K$,
\begin{equation}\label{B_s}
\cB(\al):=\{\omega_{i, j}(\al): 1\leq i \leq m-1, \, 0\leq j \leq \be_0(i)-1\}
\end{equation}
is a set of holomorphic differentials on $K(\cX)$. In particular, for every $1\leq s\leq r$ with $(m, \la_s)=1$ it follows that
\begin{align*}
\left(\omega_{i, j}(\al_s)\right)_{K(\cX)}&=\sum_{\substack{k=1\\ k\neq s}}^{r}\left(\frac{m\left(1+\floor*{\frac{i\la_k}{m}}-\ceil*{\frac{i\la_k}{m}}\right)+t_k(i)}{(m, \la_k)}-1\right)\sum_{{ Q\in \cP_{K(\cX)}, \, Q|P_{k}}}Q\\
&\quad +(mj+t_s(i)-1)Q_s+ \left(\frac{m(\beta_0(i)-1-j)+t_0(i)}{(m, \la_0)}-1\right)\sum_{{ Q\in \cP_{K(\cX)}, \, Q|P_{0}}}Q.
\end{align*}
Thus, applying Proposition \ref{prop_gap_criterion} and Theorem \ref{teo_gapsets}, %if $(m, \la_s)=1$ then
we have
$$
G(Q_s)=\{v_{Q_s}(\omega)+1: \omega\in \cB(\al_s)\}
$$
whenever $(m, \la_s)=1$.
Therefore $\cB(\al_s)$ is a basis of holomorphic differentials on $K(\cX)$ that witnesses all of the elements of the gap set $G(Q_s)$. Analogously, %if $(m, \la_0)=1$ then 
we have
$$G(Q_0)=\{v_{Q_0}(\omega)+1: \omega \in \cB(\al)\}
$$
for each $\al\in K$ whenever $(m, \la_0)=1$.
\end{remark}

Using the set of holomorphic differentials $\cB(\al)$ given in (\ref{B_s}), we can also (partially) determine gaps at the unramified and partially ramified places of the extension $K(\cX)/K(x)$. This is %summarized in 
the content of the following result.

\begin{proposition}\label{prop_gaps_2}
Let $\cX$ be the curve defined in (\ref{curveX}). The following statements hold:
\begin{enumerate}[(a)]
     \item Assume that $(m, \la_0)\neq 1$; then 
$$
\left\{\frac{mj+t_0(i)}{(m, \la_0)}: 1\leq i \leq m-1, \, 0\leq j \leq \beta_0(i)-1\right\}\subseteq G(Q) 
$$    
for any place $Q$ in $K(\cX)$ lying over $P_0$.
    \item Assume that $(m, \la_s)\neq 1$ for some $1\leq s \leq r$; then
$$
\left\{\frac{mj+t_s(i)}{(m, \la_s)}+\frac{m}{(m, \la_s)}\left(1+\floor*{\frac{i\la_s}{m}}-\ceil*{\frac{i\la_s}{m}}\right): 1\leq i \leq m-1, \, 0\leq j \leq \beta_0(i)-1\right\}\subseteq G(Q)
$$
for any place $Q$ in $K(\cX)$ lying over $P_s$.
    \item For a generic place $Q$ in $K(\cX)/K(x)$ (that is, $Q$ is a place in $K(\cX)$ that does not lie over the places $P_0, P_1, \dots, P_r$), we have
$$
\{1, 2, \dots, \be\}\subseteq G(Q)
$$
where $\be=\max\{\be_0(i): 1\leq i \leq m-1\}$.
\end{enumerate}
\end{proposition}
\begin{proof}
The result follows directly from the expression for the divisor $(\omega_{i, j}(\alpha))_{K(\mathcal{X})}$ given in Remark \ref{Remark_basis}. 
\end{proof}

It is easy to see that Theorem \ref{teo_gapsets} and Proposition \ref{prop_gaps_2} together yield the (complete) gap set at any place in the extension $K(\cX)/K(x)$ when $m=2$; that is, we obtain all gap sets of a hyperelliptic curve. On the other hand, in \cite{L1996}, it was shown that $\{1,2, \dots,m-2\}$ are gaps at a generic place in $K(\cX)/K(x)$ when $\la_1=\la_2=\cdots=\la_r=1$ and $r=m$; and in \cite[Corollary 3.7]{ABQ2019}, the authors provide conditions under which sequences of consecutive gaps at a generic place in $K(\cX)/K(x)$ exist. %In this context, 
Item $(c)$ of the preceding proposition generalizes these results. In particular, from Corollary \ref{coro1}, we conclude that whenever the curve $\cX$ defined 
in (\ref{curveX}) has at least one totally ramified place, we have
$$
\left\{1, 2, \dots , \ceil*{\frac{g(\cX)}{m-1}}\right\}\subseteq G(Q)
$$
for a generic place $Q$ in the extension $K(\cX)/K(x)$. %Now we will 
We now discuss some further consequences of Theorem \ref{teo_gapsets}.

\begin{corollary}\label{coro_G0=G} 
Suppose that $(m, \la_0)=(m, \la_s)=1$ for some $1\leq s \leq r$. We have $G(Q_0)=G(Q_s)$ whenever $\la_0+\la_s\equiv 0\bmod m$.
%Let $1\leq s_1, s_2\leq r$. If $m$ is prime then  $G(Q_{s_1})= G(Q_{s_2})$ if and only if $\la_{s_1}=\la_{s_2}$. 
\end{corollary}
\begin{proof}
From the descriptions of the gap sets $G(Q_0)$ and $G(Q_s)$ given in Theorem \ref{teo_gapsets}, we see that $G(Q_0)=G(Q_s)$ whenever $t_0(i)=t_s(i)$ for every $1\leq i\leq m-1$. On the other hand, if $\la_0+\la_s\equiv 0 \bmod m$, then $i\la_s \equiv -i\la_0 \bmod m$ for each $1\leq i \leq m-1$; and therefore 
%\begin{align*}
\[
t_s(i)=(i\la_s)\bmod m=(-i\la_0)\bmod m =m-(i\la_0)\bmod m =t_0(i).
\]
%\end{align*}
%This proves the result.
\end{proof}

In general, we have $G(Q_{s_1})=G(Q_{s_2})$ whenever $\la_{s_1}=\la_{s_2}$ for some $1\leq s_1, s_2 \leq r$ with $(m, \la_{s_1})=(m, \la_{s_2})=1$. The following example shows that the converse is not necessarily true; i.e., $G(Q_{s_1})=G(Q_{s_2})$ does not necessarily imply that $\la_{s_1}=\la_{s_2}$.

\begin{example}[Gap sets for cyclic trigonal curves]
Let $\cY$ be the cyclic trigonal curve
$$
\cY: y^3=h_1(x)h_2(x)^2
$$
where $h_1(x), h_2(x)\in K[x]$ are non-constant separable polynomials with $(h_1, h_2)=1$. Let $Q_1$ (resp. $Q_2$) denote the place in $K(\cY)$ corresponding to a root of $h_1(x)$ (resp. $h_2(x)$), $r_1:=\deg(h_1)$, $r_2:=\deg(h_2)$, and $\la_0:=r_1+2r_2$. From Theorem \ref{teo_gapsets}, we have
\begin{align*}
G(Q_1)=G(Q_2) &\Leftrightarrow \be_0(1)=\be_0(2)\\
&\Leftrightarrow  \floor*{\frac{r_1-r_2+\la_0}{3}}=\floor*{\frac{\la_0}{3}}\\
&\Leftrightarrow 0\leq r_1-r_2+(\la_0\bmod{3})\leq 2.
\end{align*}
Furthermore, Corollary \ref{coro_G0=G} establishes that whenever $(m, \la_0)=1$ we have
$$G(Q_0)=\left\{
\begin{array}{ll}
G(Q_2), & \text{if } \la_0\equiv 1 \bmod{3},\\
G(Q_1), & \text{if } \la_0\equiv 2 \bmod{3}.
\end{array}\right.
$$  
Consequently, there are at most two distinct gap sets corresponding to totally ramified places of the extension $K(\cY)/K(x)$. On the other hand, Proposition \ref{prop_gaps_2} establishes that
$$ 
\{1, 2, \dots, \max\{\be_0(1), \be_0(2)\}\}\subseteq G(Q)
$$
for a generic place $Q$ in $K(\cY)/K(x)$ or for any place $Q$ in $K(\cY)$ lying over $P_0$ when $\la_0 \equiv 0 \bmod{3}$.
\end{example}

With regard to the symmetry of the Weierstrass semigroups at totally ramified places of the extension $K(\cX)/K(x)$,\cite[Theorem 4.4]{M2023} gives (sufficient) conditions that ensure that the semigroup $H(Q_0)$ is symmetric. As another application of Theorem \ref{teo_gapsets}, we generalize {\it ibid.} by giving sufficient conditions for the semigroup $H(Q_s)$ to be symmetric, whenever $0\leq s\leq r$ is such that $(m, \lambda_s) = 1$.

\begin{corollary}\label{coro_symmetric}
The following statements hold:
\begin{enumerate}[(a)]
\item Suppose that $(m, \la_0)=1$. If
\begin{itemize}
\item $\la_k$ divides $m$ for every $k\in \{1, \dots, r\}$, or
\item $m-\la_k$ divides $m$ for every $k\in \{1, \dots, r\}$,
\end{itemize}
then $H(Q_0)$ is symmetric.
\item Suppose that $(m, \la_s)=1$ for some $1\leq s \leq r$. If
\begin{itemize}
\item $\la_k$ divides $m$ for every $k\in \{1, \dots, r\}\setminus \{s\}$ and $m-\la_0\bmod m$ divides $m$, or
\item $m-\la_k$ divides $m$ for every $k\in \{1, \dots, r\}\setminus\{s\}$, and either $\la_0\bmod m=0$ or $\la_0\bmod m$ is nonzero and divides $m$, 
\end{itemize}
then $H(Q_s)$ is symmetric.
\end{enumerate}
\end{corollary}
\begin{proof}
In each case, we will prove that $2g(\cX)-1$ is a gap.

$(a)$ %From (\ref{eq_genus}) and Theorem \ref{teo_gapsets}, we have that if 
Suppose that $\la_k$ divides $m$ for every $k\in \{1, \dots, r\}$; then Equation \eqref{eq_genus} and Theorem \ref{teo_gapsets} imply that 
$$
2g(\cX)-1=m(r-1)-\la_0=m(\be_0(1)-1)+t_0(1)\in G(Q_0).
$$
Analogously, if $m-\la_k$ divides $m$ for every $k\in \{1, \dots, r\}$, then $(m, \la_k)=(m, m-\la_k)=m-\la_k$, and therefore
$$
2g(\cX)-1=\la_0-m=m(\be_0(m-1)-1)+t_0(m-1)\in G(Q_0).
$$

$(b)$ Suppose that $\la_k$ divides $m$ for every $k\in \{1, \dots, r\}\setminus\{s\}$ and that $m-\la_0\bmod m$ divides $m$; then $(m, \la_0)=(m, \la_0\bmod m)=(m, m-\la_0\bmod m)=m-\la_0\bmod m$, and therefore 
$$
2g(\cX)-1=m\left(r-2-\floor*{\frac{\la_0}{m}}\right)+\la_s=m(\be_0(1)-1)+t_s(1)\in G(Q_s).
$$
Analogously, if $m-\la_k$ divides $m$ for every $k\in \{1, \dots, r\}\setminus\{s\}$, then $(m, \la_k)=(m, m-\la_k)=m-\la_k$. Moreover, if $\la_0\bmod m=0$ then $(m, \la_0)=m$, and therefore
$$
2g(\cX)-1=\la_0-\la_s-m=m(\be_0(m-1)-1)+t_s(m-1)\in G(Q_s);
$$
and if $0\neq \la_0\bmod m$ divides $m$ then $(m, \la_0)=(m, \la_0\bmod m)=\la_0\bmod m$, and therefore 
$$
2g(\cX)-1=m\floor*{\frac{\la_0}{m}}-\la_s=m(\be_0(m-1)-1)+t_s(m-1)\in G(Q_s).
$$
\end{proof}
Theorem \ref{teo_gapsets} also implies that $m\in H(Q_s)$ for every $0\leq s \leq r$ with $(m, \la_s)=1$. %Therefore, $m$ is an element of the Weierstrass semigroup $H(Q_s)$. 
The following result gives the $m$-Apéry set of the semigroup $H(Q_s)$; this, in turn, yields a system of generators for $H(Q_s)$. 
\begin{corollary}\label{coro2_generators}
Let $0\leq s\leq r$ be such that $(m, \la_s)=1$. The $m$-Apéry set of the Weierstrass semigroup $H(Q_s)$ is given by
    $$
    \Ap(H(Q_s), m)=\{0\}\cup\{m\beta_0(i)+t_s(i): 1\leq i \leq m-1\}.
    $$
    In particular, we have
    $$
    H(Q_s)=\langle m, m\beta_0(i)+t_s(i): 1 \leq i \leq m-1\rangle.
    $$
\end{corollary}
\begin{proof}
Theorem \ref{teo_gapsets} implies that
$$
H(Q_s)=\{mj: 0\leq j\}\cup\{mj+t_s(i): 1\leq i \leq m-1, \, \be_0(i)\le j\}.
$$
Because $\{t_s(1), \dots, t_s(m-1)\}=\{1, \dots, m-1\}$ and $m\be_0(i)+t_s(i)$ is the least element of $H(Q_s)$ congruent to $t_s(i)$ modulo $m$, it follows that 
$$
\Ap(H(Q_s), m)=\{0\}\cup\{m\beta_0(i)+t_s(i): 1\leq i \leq m-1\}.
$$
Proposition \ref{prop_Apery-generators} establishes that $\{m\}\cup (\Ap(H(Q_s), m)\setminus \{0\})$ is a system of generators for the semigroup $H(Q_s)$. Consequently
$$
H(Q_s)=\langle m, m\beta_0(i)+t_s(i): 1 \leq i \leq m-1\rangle.
$$
\end{proof}

\begin{remark}\label{remark_generators}
In the literature, generating sets and Apéry sets for the Weierstrass semigroup $H(Q_s)$ have been computed in certain cases. For example, note that whenever $(m, \la_0)=1$ we have $m\be_0(m-1)+t_0(m-1)=\la_0\in H(Q_0)$. In \cite[Theorem 3.2]{M2023}, the second author used combinatorial techniques to determine a generating set for $H(Q_0)$ based on the calculation of the Apéry set $\Ap(H(Q_0),\la_0)$. 

In \cite[Theorem 3.3]{BMNQ2025}, in order to construct families of non-isomorphic maximal curves with the same genus and automorphism group, Apéry sets $\Ap(H(Q_0), m)$ and $\Ap(H(Q_1), m)$ were computed for the curve $y^m =x^i(x^2 +1)$, where $(i,m) = (i+2,m) = 1$ and $Q_1$ (resp. $Q_0$) is the unique zero (resp. pole) of $x$. %in $K(x, y)$. 
As a result, generating sets for the semigroups $H(Q_0)$ and $H(Q_1)$ were obtained. The authors of \emph{loc. cit.} used combinatorial techniques and explicit constructions of bases for Riemann-Roch spaces. Corollary \ref{coro2_generators} extends their result.
\end{remark}

\begin{example}[Maximal curves that cannot be covered by the Hermitian curve]
Let $a, b, s \geq 1$ and $n \geq 3$ be integers for which $n$ is odd, $b$ divides $a$, and $b<a$. Let $q = p^a$ be a power of a prime $p$, let $c \in \fqs$ be such that $c^{q-1} = -1$, and suppose that $s$ divides $(q^n + 1)/(q + 1)$. Assuming these conditions, consider the following curves over $\fqsn$: 
$$
\cX_{a, b, n, s}:\quad cy^{\frac{q^n+1}{s}}=t(x)(t(x)^{q-1}+1)^{q+1}, \quad \text{where } t(x):=\textstyle\sum_{i=0}^{a/b-1}x^{p^{ib}}
$$
and 
$$
\cY_{n, s}: \quad y^{\frac{q^n+1}{s}}=(x^q+x)((x^q+x)^{q-1}-1)^{q+1}.
$$
These curves are maximal over $\fqsn$. From \cite[Theorem 3.5]{TTT2016}, the curve $\cX_{a,b,n,1}$ cannot be Galois-covered by the Hermitian curve $\cH_n$ over $\fqsn$. Furthermore, from \cite[Theorem 4.4]{TTT2016}, the curve $\cY_{3,s}$ cannot be covered by the Hermitian curve $\cH_3$ over $\fqss$ if $q >s(s+1)$.

Let $\cC$ denote either the curve $\cX_{a, b, n, s}$ or the curve $\cY_{n, s}$. Let $Q$ be a place in $K(\cC)$ corresponding to a root of $t(x)$ if $\cC=\cX_{a, b, n, s}$, or to a root of $x^q+x$ if $\cC=\cY_{n, s}$. From Corollary \ref{coro2_generators}, we have
$$
H(Q)=\left\langle M, M\left(\frac{q}{d}+\frac{q(q-1)}{d}\ceil*{\frac{i(q+1)}{M}}-\floor*{\frac{iq^3}{dM}}-1\right)+i:\,  1\leq i \leq M-1\right\rangle
$$
where $M=(q^n+1)/s$ and 
$$
d=\left\{
\begin{array}{ll}
p^b, & \text{if } \cC=\cX_{a, b, n, s},\\
1, & \text{if } \cC=\cY_{n, s}.
\end{array}\right.
$$  
This result generalizes %the one given in 
\cite[Proposition 4.1]{CB2020}, where $H(Q)$ is determined when $s=1$. 

\noindent On the other hand, note that
$$
M\be_0(M-1)+t_0(M-1)=\frac{q^3}{d}\quad \text{and}\quad M\be_0\left(\frac{qM}{q+1}\right)+t_0\left(\frac{qM}{q+1}\right)=\frac{qM}{d(q+1)} 
$$
and therefore $\left\langle M, 
\frac{q^3}{d}, \frac{qM}{d(q+1)} \right\rangle\subseteq H(Q_0)$, where $Q_0$ is the unique place at infinity of $K(\cC)$. To prove the opposite inclusion, note that for every $i\in \{1, \dots, M-1\}$ there exist integers $j_1\in \{0, \dots, q\}$ and $j_2\in \{0, \dots, M/(q+1)-1\}$ for which $i=j_1M/(q+1)+j_2$. Consequently
$$
M\be_0(i)+t_0(i)=\frac{qM}{d(q+1)}(q+1-j_1)\in \left\langle M, 
\frac{q^3}{d}, \frac{qM}{d(q+1)} \right\rangle \quad \text{if}\quad j_2=0
$$
and
$$
M\be_0(i)+t_0(i)=\frac{qM}{d(q+1)}(q-j_1)+\frac{q^3}{d}\left(\frac{M}{q+1}-j_2\right)\in \left\langle M, 
\frac{q^3}{d}, \frac{qM}{d(q+1)} \right\rangle \quad \text{if}\quad j_2\neq 0.
$$
We thus conclude that 
$$
H(Q_0)=\left\langle M, 
\frac{q^3}{d}, \frac{qM}{d(q+1)} \right\rangle.
$$
As expected, this %description of $H(Q_0)$ 
matches %the result given in 
\cite[Proposition 5.1]{TTT2016}.  
\end{example}

%With respect to some invariants of the semigroup $H(Q_s)$,
Certain invariants of the semigroup $H(Q_s)$ may be read off directly from the preceding results.
From Corollary \ref{coro2_generators}, we deduce that whenever $0\leq s \leq r$ is such that $(m, \la_s)=1$, $H(Q_s)$ is of multiplicity
\begin{equation}\label{eq_multiplicity}
m_{H(Q_s)}=\min \{m, m\al+t_s(i): 1\leq i \leq m-1, \, \be_0(i)=\al\}
\end{equation}
where $\al:=\min\{\be_0(i): 1\leq i \leq m-1\}$. On the other hand, Theorem \ref{teo_gapsets} implies that $H(Q_s)$ has Frobenius number
\begin{equation}\label{eq_frobenius}
F_{H(Q_s)}=\max\{m(\be-1)+t_s(i): 1\leq i \leq m-1, \, \be_0(i)=\be\},
\end{equation}
where $\be:=\max\{\be_0(i): 1\leq i \leq m-1\}$. 

The following result provides necessary and sufficient conditions for the multiplicity of the semigroup $H(Q_s)$ to be exactly $m$.
\begin{corollary}\label{coro3_multiplicidade}
Let $Q$ be a totally ramified place of the extension $K(\cX)/K(x)$. We have 
$$
m_{H(Q)}=m \quad \text{if and only if}\quad \beta_0(i) \geq 1 \text{ for }i=1, 2, \dots, m-1.
$$
In particular, if $m_{H(Q_s)}=m$ for \emph{some} $0\leq s \leq r$ such that $(m, \la_s)=1$ then $m_{H(Q_s)}=m$ for \emph{every} $0\leq s \leq r$ with $(m, \la_s)=1$. 
\end{corollary}
\begin{proof}
The result follows directly from Equation \eqref{eq_multiplicity}.
\end{proof}

We will now establish some additional properties of the Weierstrass semigroup at totally ramified places of $K(\cX)/K(x)$ assuming that $(m, \la_k)=1$ for every $0\leq k \leq r$. To do so, we require some preliminary results.

\begin{lemma}\label{lema2}
Suppose that $(m, \la_k)=1$ for every $1\leq k \leq r$. Then 
$$
\sum_{k=0}^{r}t_k(i)=m(r-\beta_0(i))\quad \text{for}\quad i=1, 2, \dots, m-1.
$$ 
\end{lemma}
\begin{proof} 
It suffices to note that, for every $1\leq i \leq m-1$,
\begin{align*}
\sum_{k=0}^{r}t_k(i)&=m-\left(i\la_0-\floor*{\frac{i\la_0}{m}}m\right)+\sum_{k=1}^{r}\left(i\la_k-\floor*{\frac{i\la_k}{m}}m\right)\\
&=m+\floor*{\frac{i\la_0}{m}}m-m\sum_{k=1}^{r}\left(\ceil*{\frac{i\la_k}{m}}-1\right)\\
&=m+\floor*{\frac{i\la_0}{m}}m-m\sum_{k=1}^{r}\ceil*{\frac{i\la_k}{m}}+mr\\
&=mr-m\left(\sum_{k=1}^{r}\ceil*{\frac{i\la_k}{m}}-\floor*{\frac{i\la_0}{m}}-1\right)\\
&=m(r-\beta_0(i)).
\end{align*}
\end{proof}

\begin{lemma}\label{lema3}
Suppose that $(m, \la_k)=1$ for every $0\leq k \leq r$. Then
\begin{equation*}
\beta_0(i)+\beta_0(m-i)=r-1\quad \text{for }\quad i=1, 2, \dots, m-1.
\end{equation*}
\end{lemma}
\begin{proof}
It suffices to note that, for every $1\leq i \leq m-1$, 
\begin{align*}
\beta_0(m-i)&=\sum_{k=1}^{r}\ceil*{\frac{(m-i)\la_k}{m}}-\floor*{\frac{(m-i)\la_0}{m}}-1\\
&=\sum_{k=1}^{r}\left(\la_k-\ceil*{\frac{i\la_k}{m}}+1\right)-\la_0+\floor*{\frac{i\la_0}{m}}\\
&=\floor*{\frac{i\la_0}{m}}-\sum_{k=1}^{r}\ceil*{\frac{i\la_k}{m}}+r\\
&=r-1-\left(\sum_{k=1}^{r}\ceil*{\frac{i\la_k}{m}}-\floor*{\frac{i\la_0}{m}}-1\right)\\
&=r-1-\beta_0(i).
\end{align*}
\end{proof}

\begin{proposition}\label{prop_mult-frob}
Suppose that $(m, \la_k)=1$ for each $0\leq k\leq r$. Then
$$
m_{H(Q)}=\min\{m, m(r-1)-F_{H(Q)}\}
$$
for any totally ramified place $Q$ of the extension $K(\cX)/K(x)$.
\end{proposition}
\begin{proof}
From Lemma \ref{lema3} we know that $\al+\be=r-1$, where $\al=\min\{\be_0(i): 1\leq i \leq m-1\}$ and $\be=\max\{\be_0(i): 1 \leq i \leq m-1\}$. Therefore, if $Q=Q_s$ for some $0\leq s\leq r$ then, combining (\ref{eq_multiplicity}), (\ref{eq_frobenius}), and Lemma \ref{lema3}, we obtain
\begin{align*}
m_{H(Q)}&=\min\{m, m\al+t_s(i): 1\leq i \leq m-1, \, \be_0(i)=\al\}\\
&=\min\{m, \min\{m\al+t_s(i): 1\leq i \leq m-1, \, \be_0(i)=\al\}\}\\
&=\min\{m, \min\{m(r-1-\be)+m-t_s(m-i): 1\leq i \leq m-1, \, \be_0(m-i)=\be\}\}\\
&=\min\{m, \min\{m(r-1-\be)+m-t_s(i): 1\leq i \leq m-1, \, \be_0(i)=\be\}\}\\
&=\min\{m, \min\{m(r-1)-m(\be-1)-t_s(i): 1\leq i \leq m-1, \, \be_0(i)=\be\}\}\\
&=\min\{m, m(r-1)+\min\{-m(\be-1)-t_s(i): 1\leq i \leq m-1, \, \be_0(i)=\be\}\}\\
&=\min\{m, m(r-1)-\max\{m(\be-1)+t_s(i): 1\leq i \leq m-1, \, \be_0(i)=\be\}\}\\
&=\min\{m, m(r-1)-F_{H(Q)}\}.
\end{align*}
\end{proof}
The preceding result establishes a direct relationship between the multiplicity of the semigroup $H(Q)$ and its Frobenius number. This result generalizes %the one presented in 
\cite[Proposition 4.6 iii)]{M2023}, %where its validity was proven for the specific case it was established for 
which addresses the case $Q=Q_0$.
\begin{corollary}\label{coro_multiplicity2}
Suppose that $(m, \la_k)=1$ for each $0\leq k \leq r$. Assume $m\leq r$; then $m_{H(Q)}=m$ for any totally ramified place $Q$ of the extension $K(\cX)/K(x)$.
\end{corollary}
\begin{proof}
Since $(m, \la_k)=1$ for each $0\leq k \leq r$ we have $g(\cX)=(m-1)(r-1)/2$. Therefore, if $m\leq r$ then
$$
F_{H(Q)}\leq 2g(\cX)-1=m(r-1)-r\leq m(r-1)-m
$$
for any totally ramified place $Q$ of $K(\cX)/K(x)$. This implies $m\leq m(r-1)-F_{H(Q)}$. From Proposition \ref{prop_mult-frob}, we conclude that $m_{H(Q)}=\min\{m, m(r-1)-F_{H(Q)}\}=m$.
\end{proof}

To finish this section, we will completely characterize the symmetric Weierstrass semigroups at totally ramified places of the extension $K(\cX)/K(x)$ when $(m, \la_k)=1$ for every $0\leq k \leq r$. Let $Q$ be any totally ramified place of the extension $K(\cX)/K(x)$ and let $\Ap(H(Q), m)=\{0=a_0<a_1<a_2<\cdots<a_{m-1}\}$ be the $m$-Apéry set of $H(Q)$. From Corollary \ref{coro2_generators}, there is $i_1\in\{1, 2, \dots, m-1\}$ such that $a_1=m\be_0(i_1)+t_s(i_1)$, where $s$ is such that $Q=Q_s$. Since $\be_0(i)+\be_0(m-i)=r-1$ and $t_s(i)+t_s(m-i)=m$ for every $1\leq i \leq m-1$, we deduce that $a_{m-1}=m\be_0(m-i_1)+t_s(m-i_1)$ and therefore
$$
a_1+a_{m-1}=m\be_0(i_1)+t_s(i_1)+m\be_0(m-i_1)+t_s(m-i_1)=mr.
$$
Analogously, there is $i_2\in \{1, 2, \dots, m-1\}\setminus \{i_1\}$ such that $a_2=m\be_0(i_2)+t_s(i_2)$. With the same argument, we deduce that $a_{m-2}=m\be_0(m-i_2)+t_s(m-i_2)$ and therefore
$$
a_2+a_{m-2}=m\be_0(i_2)+t_s(i_2)+m\be_0(m-i_2)+t_s(m-i_2)=mr.
$$ 
By continuing this process, we obtain that
\begin{equation}\label{eq_mr}
a_i+a_{m-i}=mr\quad \text{for}\quad i=1, 2, \dots, m-1.
\end{equation}
These relations will be used to characterize the symmetry of the Weierstrass semigroup $H(Q)$ in terms of its generators.
\begin{theorem}\label{teo_sym1}
Suppose that $(m, \la_k)=1$ for every $0\leq k\leq r$ and let $Q$ be any totally ramified place of the extension $K(\cX)/K(x)$. Then
$$
H(Q)\text{ is symmetric if and only if }H(Q)=\langle m, r \rangle.
$$ 
\end{theorem}
\begin{proof} 
It is clear that if $H(Q)=\langle m, r \rangle$ then $H(Q)$ is symmetric. Conversely, suppose that $H(Q)$ is symmetric and let $\Ap(H(Q), m)=\{0=a_0<a_1<a_2<\cdots<a_{m-1}\}$  be the $m$-Apéry set of $H(Q)$. From Proposition \ref{prop_Apery-symmetric} and (\ref{eq_mr}) we deduce that 
$$
a_{m-i}-a_{m-i-1}=mr-a_{m-1}=a_1\quad \text{for}\quad i=1, 2, \dots, m-1$$ 
or, equivalently, $a_i=a_1+a_{i-1}$ for $i=1, 2, \dots, m-1$. This implies that $a_i=ia_1$ for $i=1, 2, \dots, m-1$. Since $a_{m-1}=mr-a_1$ we obtain $a_{m-1}=mr-a_1=(m-1)a_1$, and therefore $a_1=r$. Finally, from Proposition \ref{prop_Apery-generators}, $\{m\}\cup (\Ap(H(Q), m)\setminus\{0\})=\{m, r, 2r, \dots, (m-1)r\}$ is a system of generators for the Weierstrass semigroup $H(Q)$. Therefore, $H(Q)=\langle m, r \rangle$.
\end{proof}
Complementing the previous result, the following theorem characterizes the symmetry of the Weierstrass semigroup $H(Q)$ in terms of the integers $\la_1, \la_2, \dots, \lambda_r$.

\begin{theorem}\label{teo_sym2}
Suppose that $(m, \la_k)=1$ for every $0\leq k\leq r$, and let $Q$ be any totally ramified place of the extension $K(\cX)/K(x)$. The following statements hold:
\begin{enumerate}[(a)]
\item If $Q=Q_0$ then,
$H(Q)$ is symmetric if and only if $\la_1=\la_2=\cdots=\la_r$.
\item If $Q=Q_s$ for some $1\leq s\leq r$, then
$H(Q)$ is symmetric if and only if 
\[
\la_1=\cdots=\la_{s-1}=\la_{s+1}=\cdots=\la_r=\la
\]
where $\la\in \{1, 2, \dots, m-1\}$ is the minimal positive integer for which $\la_0+\la\equiv 0 \bmod m$.
\end{enumerate}
\end{theorem}
\begin{proof}
Note first that if $\la_1=\la_2=\cdots=\la_r$, then $\la_0=\la_1r$ and
\begin{align*}
m\be_0(i)+t_0(i)&=m\left(r\ceil*{\frac{i\la_1}{m}}-\floor*{\frac{i\la_1r}{m}}-1\right)+m-(i\la_1r)\bmod m\\
&=mr\ceil*{\frac{i\la_1}{m}}-\floor*{\frac{i\la_1r}{m}}m-(i\la_1r)\bmod m\\
&=mr+mr\floor*{\frac{i\la_1}{m}}-i\la_1r\\
&=r\left(m+m\floor*{\frac{i\la_1}{m}}-i\la_1\right)\\
&=r(m-t_1(i)).
\end{align*}
Corollary \ref{coro2_generators} then implies that $H(Q_0)=\langle m, r(m-t_1(i)): 1\leq i \leq m-1\rangle=\langle m, r\rangle$ and consequently $H(Q_0)$ is symmetric. This proves the first part of item $(a)$.

To prove the first part of item $(b)$, suppose that $\la_1=\cdots=\la_{s-1}=\la_{s+1}=\cdots=\la_r=\la$, where $\la\in\{1, 2, \dots, m-1\}$ is such that $\la_0+\la\equiv 0 \bmod m$. We will show that $2g(\cX)-1=m(r-1)-r\in G(Q_s)$ to conclude that $H(Q_s)$ is symmetric. Note that $\la_0=\la_s+(r-1)\la$ and therefore $\la_s+r\la\equiv 0 \bmod m$. Moreover, 
since $(m, \la_s)=1$ %we conclude 
it follows that $(m,\la)=(m, r)=1$. Thus $1\leq m-r'\leq m-1$, where $r'=r\bmod m$, and therefore there exists $1\leq i \leq m-1$ with $t_s(i)=m-r'$, or equivalently, such that $i\la_s+r\equiv 0 \bmod m$. 

From the relations $i\la_s+r\equiv 0\bmod m$ and $\la_s+r\la\equiv 0 \bmod m$ we deduce $i\la\equiv 1 \bmod m$; while from $i\la\equiv 1 \bmod m$ and $\la_0+\la\equiv 0 \bmod m$ we deduce $i\la_0\equiv -1\bmod m$. Using %Therefore, from 
$$
i\la_s+r\equiv 0\bmod m, \quad i\la\equiv 1 \bmod m, \quad \text{and}\quad i\la_0\equiv -1\bmod m
$$
we now obtain
$$
\ceil*{\frac{i\la_s}{m}}=\frac{i\la_s+r'}{m}, \quad \ceil*{\frac{i\la}{m}}=\frac{i\la-1}{m}+1, \quad \text{and}\quad \floor*{\frac{i\la_0}{m}}=\frac{i\la_0+1}{m}-1
$$
respectively. It follows %This implies 
that
\begin{align*}
\be_0(i)&=\ceil*{\frac{i\la_s}{m}}+(r-1)\ceil*{\frac{i\la}{m}}-\floor*{\frac{i\la_0}{m}}-1\\
&=\frac{i\la_s+r'}{m}+\frac{(r-1)(i\la-1)}{m}+r-1-\frac{i\la_0+1}{m}+1-1\\
&=r-1+\frac{r'-r+i(\la_s+(r-1)\la-\la_0)}{m}\\
&=r-1+\frac{r'-r}{m}\\
&=r-1-\floor*{\frac{r}{m}}
\end{align*}
and 
$$
m(\be_0(i)-1)+t_s(i)=m\left(r-2-\floor*{\frac{r}{m}}\right)+m-r'=m(r-1)-\floor*{\frac{r}{m}}m-r\bmod m=m(r-1)-r.
$$
From Theorem \ref{teo_gapsets}, we conclude $2g(\cX)-1=m(r-1)-r=m(\be_0(i)-1)+t_s(i)\in G(Q_s)$. Thus $H(Q_s)$ is symmetric.

To finish, we address %will prove 
the second parts of items $(a)$ and $(b)$. Accordingly, suppose that $H(Q)$ is symmetric. From Theorem \ref{teo_sym1}, we have $H(Q)=\langle m, r\rangle$, and consequently $(m, r)=1$. %Since $(m, \la_k)=1$ for every $0\leq k \leq r$, we obtain 
It follows that $g(\cX)=(m-1)(r-1)/2$ and therefore
$$
2g(\cX)-1=m(r-1)-r=m(r-1)-(\floor*{r/m}m+r')=m(r-2-\floor*{r/m})+m-r'
$$
where $r'=r\bmod m$ and $1\leq m-r' \leq m-1$. Thus, since $H(Q)$ is symmetric we have that $m(r-2-\floor*{r/m})+m-r'\in G(Q)$. From Theorem \ref{teo_gapsets}, there exists $1\leq i \leq m-1$ such that $\be_0(i)=r-1-\floor*{r/m}$ and $t_s(i)=m-r'$, where $s$ is such that $Q=Q_s$. Now, from Lemma \ref{lema2},
$$
\sum_{\substack{k=0\\ k\neq s}}^{r}t_k(i)=m(r-\be_0(i))-t_s(i)=m\left(1+\floor*{\frac{r}{m}}\right)-m+r'=\floor*{\frac{r}{m}}m+r\bmod{m}=r.
$$

{\bf Case: $s=0$.} Since $1\leq t_k(i)$ for each $0\leq k\leq r$ we conclude $t_1(i)=t_2(i)=\cdots=t_r(i)=1$, or equivalently, 
$$
i\la_k\equiv 1 \bmod m\quad  \text{for each}\quad  k\in \{1, 2, \dots, r\}. 
$$
This implies that $\la_1=\la_2=\cdots=\la_r$.

{\bf Case: $1\leq s\leq r$.} Similarly, since $1\leq t_k(i)$ for each $0\leq k \leq r$ we conclude that $t_k(i)=1$ for each $0\leq k\leq r$ such that $k\neq s$. This implies
$$
i\la_k\equiv 1\bmod m \quad \text{for each} \quad k\in \{1, 2, \dots, r\}\setminus \{s\}\quad \text{ and}\quad (i\la_0)\bmod m =m-1.  
$$ 
Therefore $\la:=\la_1=\cdots=\la_{s-1}=\la_{s+1}=\cdots=\la_r$ and $i\la_0\equiv -1\bmod m$. Moreover, since $i\la\equiv 1 \bmod m$ we conclude that %$\la$ satisfies 
$\la_0+\la\equiv 0 \bmod m$. 
\end{proof}
The results given in Theorems \ref{teo_sym1} and \ref{teo_sym2} provide a complete characterization of symmetric Weierstrass semigroups at totally ramified places of the extension $K(\cX)/K(x)$ when $(m, \la_k)=1$ for every $0\leq k \leq r$. %Additionally, 
These results also generalize \cite[Theorem 4.7]{M2023}, where the same %such 
characterization is proven %for the case 
assuming that $Q=Q_0$ and $r<m$.

\section{Weierstrass places on Kummer extensions}

Let $\cX$ be an algebraic curve of genus $g$ defined over $K$, and let $K(\cX)$ be its function field. If $\char(K) = 0$ or $\char(K)\geq 2g-2$ then $K(\cX)$ is classical and, for each place $Q\in \cP_{K(\cX)}$, the Weierstrass weight of $Q$ is given by $\wt (Q)=\sum_{s\in G(Q)}s-\frac{g(g+1)}{2}$, see \cite[Corollary 7.61]{HKT2008}. In addition, a place $Q\in \cP_{K(\cX)}$ is called a Weierstrass place if $\wt(Q)>0$, and it is well known that $\sum_{Q\in \cP_{K(\cX)}}\wt (Q)=g^3-g$. In general, identifying Weierstrass places and determining their weights are central challenges in the study of algebraic curves.

Throughout this section, we assume that $\cX$ is the algebraic curve of genus $g$ defined by the affine equation given in (\ref{curveX}) and that $\char(K) = 0$ or $\char(K)\geq 2g-2$. One approach to studying the Weierstrass places of the curve $\cX$ is to analyze the asymptotic behavior of the ratio $\BW/(g^3-g)$ when $r$ goes to infinity, where $\BW$ is the sum of the Weierstrass weights of all totally ramified places of the extension $K(\cX)/K(x)$. In \cite[Corollary 10]{T1996}, Towse established that 
$$
\frac{m+1}{3(m-1)^2}\leq \liminf_{r\to \infty}\frac{\BW}{g^3-g}
$$
when $\la_1=\la_2=\dots=\la_r=1$, and that the equality holds when $(m, r)=1$. In \cite[Corollary 6.3]{ABQ2019}, Towse's result was generalized by showing that the equality holds without assuming $(m, r)=1$.

Our goal in this section is to generalize Towse's result by providing an explicit formula for the limit $\lim_{r\to \infty}\BW/(g^3-g)$ when $(m, \la_k)=1$ for every $1\leq k \leq r$. To achieve this, let $\Lambda:=\{j\in \N: 1\leq j \leq m-1, \, (m, j)=1\}$ and, for each $j\in \Lambda$, let $r_j$ be the number of distinct roots of the polynomial $f(x)$ given in (\ref{curveX}) with multiplicity $j$. Then we have $\sum_{j\in \Lambda}r_j=r$ and $\la_0=\sum_{j\in \Lambda}jr_j$. Furthermore, suppose that $\lim_{r \to \infty}r_j/r=k_j$ for each $j\in \Lambda$. Therefore $0\leq k_j  \leq 1$ and $\sum_{j\in \Lambda}k_j=1$.

\begin{theorem}\label{teo_limit}
Let $\cX$ be the curve of genus $g$ given in (\ref{curveX}) with $(m, \la_k)=1$ for every $1\leq k \leq r$ and assume that $\char(K)=0$ or $\char(K)\geq 2g-2$. Then
$$
\lim_{r\rightarrow\infty}\frac{\BW}{g^3-g}=\frac{4}{m(m-1)^3}\sum_{i=1}^{m-1}\left[\sum_{j\in \Lambda}a_{ij}k_j\right]^2-\frac{1}{m-1},
$$
where $a_{ij}:=(ij)\bmod m$.
\end{theorem}
\begin{proof}
Since the genus of $\cX$ is $g=((m-1)r+2-m-(m, \la_0))/2$ we have that 
\begin{equation}\label{eq_limit_1}
g^3-g=\frac{(m-1)^3}{8}r^3+\O(r^2),
\end{equation}
where $\O(r^2)$ denotes a polynomial in the variable $r$ of degree less or equal than $2$. On the other hand, from the explicit description of the gap sets $G(Q_s)$ for $1\leq s\leq r$ given in Theorem \ref{teo_gapsets} and from Lemma \ref{lema2}, we have that, regardless of whether $Q_0$ is a totally ramified place of $K(\cX)/K(x)$,
\begin{align*}
\BW&=\sum_{s=1}^{r}\wt(Q_s)+\O(r^2)\\
&=\sum_{s=1}^{r}\sum_{i=1}^{m-1}\sum_{j=0}^{\beta_0(i)-1}(mj+t_s(i))-\frac{rg(g+1)}{2}+\O(r^2)\\
&=\sum_{i=1}^{m-1}\sum_{j=0}^{\beta_0(i)-1}\sum_{s=1}^{r}(mj+t_s(i))-\frac{rg(g+1)}{2}+\O(r^2)\\
&=\sum_{i=1}^{m-1}\sum_{j=0}^{\beta_0(i)-1}\left(mjr+m(r-\be_0(i))-t_0(i)\right)-\frac{rg(g+1)}{2}+\O(r^2)\\
&=m\sum_{i=1}^{m-1}\sum_{j=0}^{\beta_0(i)-1}\left(r(j+1)-\beta_0(i)\right)-\sum_{i=1}^{m-1}\sum_{j=0}^{\beta_0(i)-1}t_0(i)-\frac{rg(g+1)}{2}+\O(r^2)\\
&=m\sum_{i=1}^{m-1}\left(\frac{r\be_0(i)(\be_0(i)+1)}{2}-\beta_0(i)^2\right)-\sum_{i=1}^{m-1}t_0(i)\be_0(i)-\frac{rg(g+1)}{2}+\O(r^2)\\
&=m\sum_{i=1}^{m-1}\left(\frac{(r-2)\be_0(i)^2+r\be_0(i)}{2}\right)-\sum_{i=1}^{m-1}t_0(i)\be_0(i)-\frac{rg(g+1)}{2}+\O(r^2)\\
&=\frac{m(r-2)}{2}\sum_{i=1}^{m-1}\be_0(i)^2+\frac{mr}{2}\sum_{i=1}^{m-1}\be_0(i)-\sum_{i=1}^{m-1}t_0(i)\be_0(i)-\frac{rg(g+1)}{2}+\O(r^2)\\
&=\frac{m(r-2)}{2}\sum_{i=1}^{m-1}\be_0(i)^2+\frac{mrg}{2}-\sum_{i=1}^{m-1}t_0(i)\be_0(i)-\frac{rg(g+1)}{2}+\O(r^2).
\end{align*}
Moreover, from Corollary \ref{coro1} and since $1\leq t_0(i)\leq m$ for every $1\leq i\leq m-1$, we obtain that $g\leq \sum_{i=1}^{m-1}t_0(i)\be_0(i)\leq mg$ and therefore
$$
\frac{mrg}{2}-\sum_{i=1}^{m-1}t_0(i)\be_0(i)-\frac{rg(g+1)}{2}=-\frac{(m-1)^2}{8}r^3+\O(r^2).
$$
This implies that
\begin{equation}\label{eq_limit_2}
\BW=\frac{m(r-2)}{2}\sum_{i=1}^{m-1}\beta_0(i)^2-\frac{(m-1)^2}{8}r^3+\O(r^2).    
\end{equation}
Now, note that 
$$
\beta_0(i)=\sum_{k=1}^{r}\ceil*{\frac{i\la_k}{m}}-\floor*{\frac{i\la_0}{m}}-1=\sum_{j\in \Lambda}r_j\ceil*{\frac{ij}{m}}-\floor*{\frac{i\la_0}{m}}-1=\frac{r}{m}\left(\sum_{j\in\Lambda}m\frac{r_j}{r}\ceil*{\frac{ij}{m}}\right)-\floor*{\frac{i\la_0}{m}}-1.
$$
Furthermore, since $\la_0=\sum_{j\in\Lambda }jr_j$ we have
$$
\frac{i\la_0}{m}=\frac{r}{m}\left(\sum_{j\in \Lambda}ij\frac{r_j}{r}\right)
$$
and therefore
$$
\frac{r}{m}\left(\sum_{j\in \Lambda}ij\frac{r_j}{r}\right)-1<\floor*{\frac{i\la_0}{m}}\leq \frac{r}{m}\left(\sum_{j\in \Lambda}ij\frac{r_j}{r}\right).
$$
Thus,
$$
\frac{r}{m}\sum_{j\in \Lambda}\frac{r_j}{r}\left(m\ceil*{\frac{ij}{m}}-ij\right)-1\leq \beta_0(i)<\frac{r}{m}\sum_{j\in \Lambda}\frac{r_j}{r}\left(m\ceil*{\frac{ij}{m}}-ij\right).
$$
Moreover, since $(m, j)=1$ for each $j\in \Lambda$, we obtain 
$$
m\ceil*{\frac{ij}{m}}-ij=m\left(\floor*{\frac{ij}{m}}+1\right)-ij=m-\left(ij-m\floor*{\frac{ij}{m}}\right)=m-a_{ij}
$$
and hence
\begin{equation*}\label{equation_des_be0}
\frac{r}{m}\sum_{j\in \Lambda}\frac{r_j}{r}\left(m-a_{ij}\right)-1\leq \beta_0(i)<\frac{r}{m}\sum_{j\in \Lambda}\frac{r_j}{r}\left(m-a_{ij}\right).    
\end{equation*}
For $1\leq i \leq m-1$ define $\gamma_i(r):=\sum_{j\in \Lambda}\frac{r_j}{r}\left(m-a_{ij}\right)$. Assuming $m\leq r$ we have that
$$
0\leq \frac{r}{m}\gamma_i(r)-1\leq \beta_0(i)<\frac{r}{m}\gamma_i(r)
$$
and therefore
\begin{equation}\label{eq_des_b2}
\frac{r^2}{m^2}\gamma_i(r)^2-\frac{2r}{m}\gamma_i(r)+1\leq\beta_0(i)^2<\frac{r^2}{m^2}\gamma_i(r)^2.
\end{equation}
From (\ref{eq_limit_1}), (\ref{eq_limit_2}), and (\ref{eq_des_b2}), and since $\lim_{r\to \infty}\gamma_i(r)=\sum_{j\in \Lambda}(m-a_{ij})k_j$, we obtain
\begin{align*}
\lim_{r\to \infty}\frac{\BW}{g^3-g}&\leq \lim_{r\to \infty}\frac{\displaystyle\frac{(r-2)r^2}{2m}\sum_{i=1}^{m-1}\gamma_i(r)^2-\frac{(m-1)^2}{8}r^3+\O(r^2)}{\displaystyle\frac{(m-1)^3}{8}r^3+\O(r^2)}\\
&=\lim_{r\to \infty}\frac{\displaystyle\frac{(r-2)}{2mr}\sum_{i=1}^{m-1}\gamma_i(r)^2-\frac{(m-1)^2}{8}+\frac{\O(r^2)}{r^3}}{\displaystyle\frac{(m-1)^3}{8}+\frac{\O(r^2)}{r^3}}\\
&=\frac{\displaystyle\frac{1}{2m}\sum_{i=1}^{m-1}\left[\sum_{j\in\Lambda}(m-a_{ij})k_j\right]^2-\frac{(m-1)^2}{8}}{\displaystyle\frac{(m-1)^3}{8}}\\
&=\frac{4}{m(m-1)^3}\sum_{i=1}^{m-1}\left[\sum_{j\in\Lambda}(m-a_{ij})k_j\right]^2-\frac{1}{m-1}.
\end{align*}
Analogously,
\begin{align*}
&\lim_{r\to \infty}\frac{\BW}{g^3-g}\\
&\geq \lim_{r\to \infty}\frac{\displaystyle\frac{(r-2)r^2}{2m}\sum_{i=1}^{m-1}\gamma_i(r)^2-r(r-2)\sum_{i=1}^{m-1}\gamma_i(r)+\frac{m(m-1)(r-2)}{2}-\frac{(m-1)^2}{8}r^3+\O(r^2)}{\displaystyle\frac{(m-1)^3}{8}r^3+\O(r^2)}\\
&=\lim_{r\to \infty}\frac{\displaystyle\frac{(r-2)}{2mr}\sum_{i=1}^{m-1}\gamma_i(r)^2-\frac{(r-1)}{r^2}\sum_{i=1}^{m-1}\gamma_i(r)+\frac{m(m-1)(r-2)}{2r^3}-\frac{(m-1)^2}{8}+\frac{\O(r^2)}{r^3}}{\displaystyle\frac{(m-1)^3}{8}+\frac{\O(r^2)}{r^3}}\\
&=\frac{\displaystyle\frac{1}{2m}\sum_{i=1}^{m-1}\left[\sum_{j\in\Lambda}(m-a_{ij})k_j\right]^2-\frac{(m-1)^2}{8}}{\displaystyle\frac{(m-1)^3}{8}}=\frac{4}{m(m-1)^3}\sum_{i=1}^{m-1}\left[\sum_{j\in\Lambda}(m-a_{ij})k_j\right]^2-\frac{1}{m-1}.
\end{align*}
Thus, we conclude that 
$$
\lim_{r\to \infty}\frac{\BW}{g^3-g}=\frac{4}{m(m-1)^3}\sum_{i=1}^{m-1}\left[\sum_{j\in\Lambda}(m-a_{ij})k_j\right]^2-\frac{1}{m-1}.
$$
To conclude the proof, observe that $a_{ij}+a_{(m-i)j}=m$ for every $1\leq i \leq m-1$ and $j\in \Lambda$. Therefore,
$$
\sum_{i=1}^{m-1}\left[\sum_{j\in\Lambda}(m-a_{ij})k_j\right]^2=\sum_{i=1}^{m-1}\left[\sum_{j\in\Lambda}a_{(m-i)j}k_j\right]^2=\sum_{i=1}^{m-1}\left[\sum_{j\in\Lambda}a_{ij}k_j\right]^2.
$$
%$$
%\sum_{i=1}^{m-1}\left[\sum_{j\in\Lambda}(m-a_{ij})k_j\right]^2=\sum_{j\in \Lambda}k_j^2\left(\sum_{i=1}^{m-1}(m-a_{ij})^2\right)+2\sum_{\substack{j_1, j_2 \in \Lambda \\j_1<j_2}}k_{j_1}k_{j_2}\left(\sum_{i=1}^{m-1}(m-a_{ij_1})(m-a_{ij_2})\right).
%$$
%Since 
%$$
%\sum_{i=1}^{m-1}(m-a_{ij})^2=\sum_{i=1}^{m-1}a_{ij}^2\quad \text{ and }\quad\sum_{i=1}^{m-1}(m-a_{ij_1})(m-a_{ij_2})=\sum_{i=1}^{m-1}a_{ij_1}a_{ij_2},
%$$
%we conclude that 
%$$
%\sum_{i=1}^{m-1}\left[\sum_{j\in\Lambda}(m-a_{ij})k_j\right]^2=\sum_{i=1}^{m-1}\left[\sum_{j\in\Lambda}a_{ij}k_j\right]^2.
%$$
\end{proof}

Based on the formula provided in the previous theorem, we establish an upper and a lower bound for the limit $\lim_{r\to \infty}\BW/(g^3-g)$. Furthermore, we present the conditions under which this limit attains these bounds.

\begin{proposition}\label{prop_upperbound}
Under the assumptions of Theorem \ref{teo_limit}, we have that
$$
\lim_{r\to \infty}\frac{\BW}{g^3-g}\leq \frac{m+1}{3(m-1)^2}.
$$
Equality holds if and only if $k_j=1$ for some $j\in \Lambda$.
\end{proposition}
\begin{proof}
Note that if there exists $j\in \Lambda$ such that $k_j\neq 0$ and $k_i=0$ for $i\neq j$, then $k_j=1$ and in this case we obtain
\begin{align*}
\lim_{r\to \infty}\frac{\BW}{g^3-g}&=\frac{4}{m(m-1)^3}\sum_{i=1}^{m-1}a_{ij}^2-\frac{1}{m-1}\\
&=\frac{4}{m(m-1)^3}\sum_{i=1}^{m-1}i^2-\frac{1}{m-1}\\
&=\frac{2(2m-1)}{3(m-1)^2}-\frac{1}{m-1}\\
&=\frac{m+1}{3(m-1)^2}.
\end{align*}
Now, if there exist at least two distinct elements $j_1, j_2$ in $\Lambda$ such that $k_{j_1}\neq 0$ and $k_{j_2}\neq 0$ then, from the strict convexity of the function $h(x)=x^2$ and the fact that $\sum_{j\in \Lambda}k_j=1$, we obtain
\begin{align*}
\lim_{r\to \infty}\frac{\BW}{g^3-g}&=\frac{4}{m(m-1)^3}\sum_{i=1}^{m-1}\left[\sum_{j\in \Lambda}a_{ij}k_j\right]^2-\frac{1}{m-1}\\
&<\frac{4}{m(m-1)^3}\sum_{i=1}^{m-1}\sum_{j\in \Lambda}a_{ij}^2k_j-\frac{1}{m-1}\\
&=\frac{4}{m(m-1)^3}\sum_{j\in \Lambda}k_j\left(\sum_{i=1}^{m-1}a_{ij}^2\right)-\frac{1}{m-1}\\
&=\frac{4}{m(m-1)^3}\sum_{j\in \Lambda}k_j\left(\sum_{i=1}^{m-1}i^2\right)-\frac{1}{m-1}\\
&=\frac{2(2m-1)}{3(m-1)^2}-\frac{1}{m-1}\\
&=\frac{m+1}{3(m-1)^2}.
\end{align*}
Thus, the result follows.
\end{proof}

This result shows that the formula given in Theorem \ref{teo_limit} generalizes the classical result given by Towse in \cite[Corollary 10]{T1996}, as well as the result provided in \cite[Corollary 6.3]{ABQ2019}. Moreover, it establishes that the maximum value that $\lim_{r\to \infty}\BW/(g^3-g)$ can attain in this case coincides with the value provided by Towse.

\begin{proposition}\label{prop_lowerbound}
Under the assumptions of Theorem \ref{teo_limit}, we have that
$$
\frac{1}{(m-1)^2}\leq \lim_{r\to \infty}\frac{\BW}{g^3-g}.
$$
In addition, if $k_j=k_{m-j}$ for each $j\in \Lambda$ then equality holds.
\end{proposition}
\begin{proof}
We begin by noting that $a_{ij}+a_{i(m-j)}=m$ for every $1\leq i \leq m-1$ and $j\in \Lambda$. Moreover, $j\in \Lambda$ if and only if $m-j\in \Lambda$. Thus,
\begin{align*}
\sum_{j\in \Lambda}a_{ij}k_j&=\frac{1}{2}\sum_{j\in \Lambda}(a_{ij}k_j+a_{i(m-j)}k_{m-j})\\
&=\frac{1}{2}\sum_{j\in \Lambda}(a_{ij}k_j+(m-a_{ij})k_{m-j})\\
&=\frac{1}{2}\sum_{j\in \Lambda}(mk_{m-j}+a_{ij}(k_j-k_{m-j}))\\
&=\frac{1}{2}\left(m+\sum_{j\in \Lambda}a_{ij}(k_j-k_{m-j})\right).
%&=\frac{1}{2}\left(\sum_{j\in \Lambda_0}(a_{ij}k_j+a_{i(m-j)}k_{m-j})+\sum_{j\in \Lambda_1}(a_{ij}k_j+a_{i(m-j)}k_{m-j})+\sum_{j\in \Lambda_2}(a_{ij}k_j+a_{i(m-j)}k_{m-j})\right)\\
%&=\frac{1}{2}\left(\sum_{j\in \Lambda_0}mk_j+\sum_{j\in \Lambda_1}(mk_{j}-a_{i(m-j)}(k_j-k_{m-j}))+\sum_{j\in \Lambda_2}(mk_j+a_{i(m-j)}(k_{m-j}-k_j))\right)\\
%&=\frac{1}{2}\left(\sum_{j\in \Lambda}mk_j-\sum_{j\in \Lambda_1}a_{i(m-j)}(k_j-k_{m-j})+\sum_{j\in \Lambda_2}a_{i(m-j)}(k_{m-j}-k_j)\right)\\
%&=\frac{1}{2}\left(m-\sum_{j\in \Lambda_1}a_{i(m-j)}(k_j-k_{m-j})+\sum_{j\in \Lambda_1}a_{ij}(k_{j}-k_{m-j})\right)\\
%&=\frac{1}{2}\left(m+\sum_{j\in \Lambda_1}(a_{ij}-a_{i(m-j)})(k_j-k_{m-j})\right).
\end{align*}
This implies that 
\begin{align*}
\sum_{i=1}^{m-1}\left[\sum_{j\in \Lambda}a_{ij}k_j\right]^2&=\sum_{i=1}^{m-1}\left(\frac{m^2}{4}+\frac{m}{2}\sum_{j\in \Lambda}a_{ij}(k_j-k_{m-j})+\frac{1}{4}\left[\sum_{j\in \Lambda}a_{ij}(k_j-k_{m-j})\right]^2\right)\\
%&=\sum_{i=1}^{m-1}\left(\frac{m^2}{4}+\frac{m}{2}\sum_{j\in \Lambda_1}(a_{ij}-a_{i(m-j)})(k_j-k_{m-j})+\frac{1}{4}\left[\sum_{j\in \Lambda_1}(a_{ij}-a_{i(m-j)})(k_j-k_{m-j})\right]^2\right)\\
&=\frac{m^2(m-1)}{4}+\frac{m}{2}\sum_{j\in \Lambda}\left[(k_j-k_{m-j})\sum_{i=1}^{m-1}a_{ij}\right]+\frac{1}{4}\sum_{i=1}^{m-1}\left[\sum_{j\in \Lambda}a_{ij}(k_j-k_{m-j})\right]^2\\
&=\frac{m^2(m-1)}{4}+\frac{m^2(m-1)}{4}\sum_{j\in \Lambda}(k_j-k_{m-j})+\frac{1}{4}\sum_{i=1}^{m-1}\left[\sum_{j\in \Lambda}a_{ij}(k_j-k_{m-j})\right]^2\\
%&=\frac{m^2(m-1)}{4}+\frac{m}{2}\sum_{j\in \Lambda_1}\left[(k_j-k_{m-j})\sum_{i=1}^{m-1}(a_{ij}-a_{i(m-j)})\right]+\frac{1}{4}\sum_{i=1}^{m-1}\left[\sum_{j\in \Lambda_1}(a_{ij}-a_{i(m-j)})(k_j-k_{m-j})\right]^2\\
&=\frac{m^2(m-1)}{4}+\frac{1}{4}\sum_{i=1}^{m-1}\left[\sum_{j\in \Lambda}a_{ij}(k_j-k_{m-j})\right]^2\\
%&=\frac{m^2(m-1)}{4}+\frac{1}{4}\sum_{i=1}^{m-1}\left[\sum_{j\in \Lambda_1}(a_{ij}-a_{i(m-j)})(k_j-k_{m-j})\right]^2\\
&\geq \frac{m^2(m-1)}{4}
\end{align*}
and hence
$$
\lim_{r\to \infty}\frac{\BW}{g^3-g}=\frac{4}{m(m-1)^3}\sum_{i=1}^{m-1}\left[\sum_{j\in \Lambda}a_{ij}k_j\right]^2-\frac{1}{m-1}\geq \frac{m}{(m-1)^2}-\frac{1}{m-1}=\frac{1}{(m-1)^2}.
$$
Finally, it is clear that if $k_j=k_{m-j}$ for each $j\in \Lambda$ then
$$
\sum_{i=1}^{m-1}\left[\sum_{j\in \Lambda}a_{ij}k_j\right]^2=\frac{m^2(m-1)}{4}\quad \text{and}\quad \lim_{r\to \infty}\frac{\BW}{g^3-g}=\frac{1}{(m-1)^2}.
$$
\end{proof}

%STYLE:
%\bibliographystyle{alpha}
%\bibliographystyle{amsalpha}
\bibliographystyle{abbrv}

\bibliography{bibcodes} 

%\begin{comment}
%\begin{thebibliography}{99}
%\end{thebibliography}
\end{document}